\documentclass[12pt]{amsart}
\usepackage{amsmath,amssymb,amsthm}
\usepackage{mathtools}
\usepackage{scalerel}
\usepackage{enumerate}
\usepackage{hyperref}
\hypersetup{colorlinks=true,citecolor=blue,linkcolor=red}
\usepackage{cite}
\usepackage{color}
\usepackage[usestackEOL]{stackengine}
\usepackage[a4paper,width=16.5cm,top=3cm,bottom=2.8cm]{geometry}

\newtheorem{theorem}{Theorem}[section]

\newtheorem{corollary}[theorem]{Corollary}
\newtheorem{lemma}[theorem]{Lemma}
\newtheorem{definition}[theorem]{Definition}
\newtheorem{remark}[theorem]{Remark}

\theoremstyle{plain}
\newtheorem*{theorem*}{Theorem}
\newtheorem*{question*}{Question}
\newtheorem*{example*}{Example}

\def\dashint{\,\ThisStyle{\ensurestackMath{%
			\stackinset{c}{.2\LMpt}{c}{.5\LMpt}{\SavedStyle-}{\SavedStyle\phantom{\int}}}%
		\setbox0=\hbox{$\SavedStyle\int\,$}\kern-\wd0}\int}

\newcommand{\zed}{\mathbb{Z}}

\newcommand{\dyprodhp}{H_d^p(\mathbb{R}\otimes\mathbb{R})}
\newcommand{\dotdyprodhp}{\dot{H}_d^p(\mathbb{R}\otimes\mathbb{R})}
\newcommand{\dyprodhpn}[1]{\|#1\|_{\dyprodhp}}
\newcommand{\dotdyprodhpn}[1]{\|#1\|_{\dotdyprodhp}}

\newcommand{\dotdyprodhq}{\dot{H}_d^q(\mathbb{R}\otimes\mathbb{R})}
\newcommand{\dotdyprodhqn}[1]{\|#1\|_{\dotdyprodhq}}

\newcommand{\dyprodhr}{H_d^r(\mathbb{R}\otimes\mathbb{R})}

\newcommand{\dotdyprodhr}{\dot{H}_d^r(\mathbb{R}\otimes\mathbb{R})}
\newcommand{\dotdyprodhrn}[1]{\|#1\|_{\dotdyprodhr}}

\newcommand{\dyprodbmo}{\text{BMO}_d(\mathbb{R}\otimes\mathbb{R})}

\newcommand{\dyprodbmon}[1]{\|#1\|_{\dyprodbmo}}

\newcommand{\avr}[2]{\left\langle#1\right\rangle_{#2}}

\newcommand{\lprtwo}{L^p(\mathbb{R}^2)}
\newcommand{\lprtwonorm}[1]{\|#1\|_{\lprtwo}}

\newcommand{\lqrtwo}{L^q(\mathbb{R}^2)}

\newcommand{\lrrtwo}{L^r(\mathbb{R}^2)}
\newcommand{\lrrtwonorm}[1]{\|#1\|_{\lrrtwo}}

\newcommand{\ltwotwo}{L^2(\mathbb{R}^2)}

\newcommand{\R}{\mathbb{R}}
\newcommand{\D}{\mathcal{D}}
\newcommand{\Dtwo}{\D\otimes\D}

\begin{document}
	
	\title{The Operator norm of paraproducts on bi-parameter Hardy spaces}
	\author{Shahaboddin Shaabani}
	\address{Department of Mathematics and Statistics\\Concordia University}
	\email{shahaboddin.shaabani@concordia.ca}	
	\date{}
	\begin{abstract}
	It is shown that for $0<p,q,r<\infty$, with $\frac{1}{q} = \frac{1}{p} + \frac{1}{r}$, the operator norm of the dyadic paraproduct of the form
	\[
	\pi_g(f) := \sum_{R \in \Dtwo} g_R \avr{f}{R} h_R,
	\]
	from the bi-parameter dyadic Hardy space $\dyprodhp$ to $\dotdyprodhq$ is comparable to $\dotdyprodhrn{g}$. We also prove that for all $0 < p < \infty$, there holds 
	\[
	\dyprodbmon{g} \simeq \|\pi_g\|_{\dyprodhp \to \dotdyprodhp}.
	\]
	Similar results are obtained for bi-parameter Fourier paraproducts of the same form.
	
	\end{abstract}
	\maketitle

\section{Introduction}

The one-parameter dyadic paraproduct operator with symbol $g$ is defined as  
\[
\pi_g(f) := \sum_{I \in \mathcal{D}} g_I \avr{f}{I} h_I,
\]
where $\mathcal{D}$ denotes the collection of all dyadic intervals on the real line, $h_I$ is the $L^2$-normalized Haar wavelet associated with the interval $I$, $g_I$ represents the Haar coefficient of $g$, and $\avr{f}{I}$ is the average of $f$ over the interval $I$. Bilinear forms (in terms of $f$ and $g$) of this type are among the most important ones in harmonic analysis, with many applications to PDEs. This is mainly due to the fact that many bilinear forms can be decomposed in terms of paraproducts and their adjoints. For instance, the product of two functions $f$ and $g$ can be written as  
\[
fg=\pi_g(f)+\pi_f(g)+\pi_g^*(f),
\]
where $\pi_g^*$ denotes the adjoint of $\pi_g$. For this reason, the boundedness properties of these operators play a crucial role in analyzing various problems in harmonic analysis and PDEs. We refer the reader to \cite{MR2682821} for a brief introduction and to \cite{MR3052499} for an excellent exposition of this subject. See also \cite{MR0518170,MR0631751,MR1853518,MR1934198,MR2381883,MR2730492} for various boundedness properties of paraproducts. Because of their importance, it is natural to wonder about the norm of these operators acting between various function spaces. The first result in this direction appeared in \cite{MR2381883}, where it was shown that  
\[
\|\pi_g\|_{L^p(\R)\to \dot{L}^p(\R)}\simeq \|g\|_{\text{BMO}_d(\R)}, \quad 1<p<\infty.
\]  
Here, $\dot{L}^p(\mathbb{R})$ is the Lebesgue space $L^p(\mathbb{R})$ modulo constants, with the quotient norm defined as
 \[
 \|f\|_{\dot{L}^p(\mathbb{R})} := \inf_{c \in \mathbb{R}} \|f - c\|_{L^p(\mathbb{R})},
 \]
 and $\text{BMO}_d(\R)$ stands for the dyadic $\text{BMO}$, the space of functions with bounded mean oscillation on the dyadic intervals of the real line. In addition, recently in \cite{hänninen2023weighted}, the authors extended the above result to the off-diagonal range of exponents. More precisely, they showed, among other things, that
 \[
 \|\pi_g\|_{L^p(\mathbb{R}) \to \dot{L}^q(\mathbb{R})} \simeq \|g\|_{\dot{L}^r(\mathbb{R})} \quad \text{where} \quad \frac{1}{q} = \frac{1}{p} + \frac{1}{r}, \quad 1 < p, q, r < \infty.
 \]
The next progress in the subject comes from our recent work in \cite{shaabani2024operator}, where our main contribution was to replace Lebesgue spaces with Hardy spaces and to lift the restrictions on the exponents on the right-hand side of the above two results. Specifically, it was shown that 
\begin{align}
        &\|\pi_g\|_{H_d^p(\mathbb{R}) \to \dot{H}_d^p(\mathbb{R})} \simeq \|g\|_{\text{BMO}_d(\mathbb{R})}, \quad 0<p<\infty,\label{mypreviousresult1}\\
    	&\|\pi_g\|_{H_d^p(\mathbb{R}) \to \dot{H}_d^q(\mathbb{R})} \simeq \|g\|_{\dot{H}_d^r(\mathbb{R})}, \quad \frac{1}{q} = \frac{1}{p} + \frac{1}{r}, \quad 0 < p, q, r < \infty,\label{mypreviousresult}
\end{align}
where the $H_d^p(\mathbb{R})$-norm is defined as the $L^p(\mathbb{R})$-norm of the dyadic maximal function, and the $\dot{H}_d^p(\mathbb{R})$-norm refers to the $L^p(\mathbb{R})$-norm of the dyadic square function. We also obtained similar results for Fourier paraproducts in the continuous setting, and we refer the reader to \cite{shaabani2024operator} for precise statements in this context.\\

The proof idea in \cite{hänninen2023weighted}, which is similar to that in \cite{MR4338459}, heavily relies on the duality of Lebesgue spaces. However, as demonstrated in \cite{shaabani2024operator}, this approach fails when $0 < q < 1$. In \cite{shaabani2024operator}, we therefore adopted a direct method. By using a suitable pointwise sparse domination of the square function of the symbol $g$, we were able to construct a test function $f$ such that, when testing the operator $\pi_g$ on $f$, we could recover the $L^r(\mathbb{R})$-norm of the square function of $g$, achieving the desired result.\\

In the present paper, we focus on operators acting on bi-parameter Hardy spaces. There are various types of bi-parameter paraproducts, and the one we study is the most similar to the one-parameter operator. It is defined as
\[
\pi_g(f) := \sum_{R \in \Dtwo} g_R \avr{f}{R} h_R,
\]
where the sum is taken over the collection of all dyadic rectangles in the plane (see the next section for precise definitions and notation). We refer the reader to \cite{MR3052499} for an exposition of multi-parameter paraproducts. See also \cite{MR2134868,MR2313844,MR2320408}. To obtain an analog of \eqref{mypreviousresult1} and \eqref{mypreviousresult} for this operator, we employ a similar strategy and demonstrate how the one-parameter arguments in \cite{shaabani2024operator} can be modified to work in the multi-parameter setting. As in our previous work, we first present our arguments in the dyadic setting and then extend them to the continuous setting. Before doing so, let us fix some definitions and notation.
	
\section{preliminaries}	
	
As mentioned before, by $\D$ we mean the collection of all dyadic intervals in $\R$, and $\Dtwo$ stands for the collection of all dyadic rectangles in the plane. For $f\in L_{loc}^1(\R^2)$ and $E$ a measurable set of finite positive measure, we denote the average of $f$ over $E$ by 
\[
\avr{f}{E}:=|E|^{-1}\int_{E}f.
\]	
For such a function, the bi-parameter dyadic maximal operator is defined as
\[
M_d(f)(x):=\sup_{\substack{x\in R\\R\in\Dtwo}}|\avr{f}{R}|.
\]
When the supremum is taken over all rectangles (not necessarily dyadic) with sides parallel to the axes, the resulting operator is denoted by $M$ and is referred to as the strong maximal operator. We also need the following version of this operator:
\[
M_s(f)(x):=M(|f|^s)^\frac{1}{s}(x), \quad 0<s<\infty,
\]
as well as the bi-parameter Fefferman-Stein inequality, which states that
\[
\lprtwonorm{\big(\sum_{j}M_s(f_j)^2\big)^{\frac{1}{2}}}\lesssim \lprtwonorm{\big(\sum_{j}|f_j|^2\big)^{\frac{1}{2}}}, \quad 0<s<p<\infty.
\]
The next notation deals with an enlargement of open sets $\Omega$, which occurs quite often in multi-parameter theory. We use the following somewhat standard notation:
\[
\tilde{\Omega}:=\big\{M(\chi_{\Omega})>\frac{1}{2}\big\},
\]
and recall that
\[
\Omega\subset \tilde{\Omega}, \quad |\tilde{\Omega}|\lesssim |\Omega|,
\]
which follows from the boundedness of $M$ on, say, $L^2(\R^2)$. See \cite{MR0864369,MR766959,MR664621} for the proof of the above assertions and other properties of maximal operators in the product setting. \\

Now, to modify our one-parameter arguments in \cite{shaabani2024operator}, we need to generalize some properties of sparse families of cubes to simple families of measurable sets.
\begin{definition}
	A sequence of measurable sets $\left\{\Omega_i\right\}_{i\ge0}$ of finite measure is called contracting if
	\[
	\Omega_{i+1}\subset \Omega_i, \quad |\Omega_{i+1}|\le\frac{1}{2}|\Omega_{i}|, \quad i=0,1,2,\ldots
	\]
\end{definition}

In the next lemma, we show that when dealing with $L^p$ norms, a contracting family can be treated as a disjoint family.\begin{lemma}\label{lpsparselemma}
	Let $\left\{\Omega_i\right\}_{i\ge0}$ be a contracting family. Then, for any sequence of nonnegative numbers $\{a_i\}_{i\ge0}$ and any $0<p<\infty$, we have
	\[
	\|\sum_{i\ge0}a_i\chi_{\Omega_i}\|_{L^p}\simeq_p \big(\sum_{i\ge0} a_i^p|\Omega_{i}|\big)^{\frac{1}{p}}.
	\]
\end{lemma}
\begin{proof}
	First, note that 
	\[
	\big(\sum_{i\ge0} a_i^p|\Omega_{i}|\big)^{\frac{1}{p}} \lesssim	\|\sum_{i\ge0}a_i\chi_{\Omega_i\backslash\Omega_{i+1}}\|_{L^p}\le \|\sum_{i\ge0}a_i\chi_{\Omega_i}\|_{L^p}.
	\]
	So it remains to prove the other direction. For $0<p\le1$, from sub-additivity we get
	\[
	\|\sum_{i\ge0}a_i\chi_{\Omega_i}\|_{L^p}^p\le \sum_{i\ge0}\|a_i\chi_{\Omega_i}\|_{L^p}^p=\sum_{i\ge0} a_i^p|\Omega_{i}|.
	\]
	Therefore, we are left to prove
\begin{equation}\label{thedesired}
			\|\sum_{i\ge0}a_i\chi_{\Omega_i}\|_{L^p}\lesssim\big(\sum_{i\ge0} a_i^p|\Omega_{i}|\big)^{\frac{1}{p}}, \quad 1<p<\infty.
\end{equation}
	To this aim, take a function $g\in L^{p'}$, where $p'$ is the H\"older's conjugate of $p$ and note that
	\[
	\int g\sum_{i\ge0}a_i\chi_{\Omega_i}\le2\sum_{i\ge0}a_i\avr{|g|}{\Omega_i}|\Omega_{i}\backslash\Omega_{i+1}|\le 2\int m(g) \sum_{i\ge0}a_i\chi_{\Omega_i\backslash\Omega_{i+1}},
	\]
	where in the above 
	\[
	m(g)(x):=\sup_{\substack{x\in\Omega_{i}\\i\ge0}} \avr{|g|}{\Omega_i}.
	\]
For now, let us assume that this operator is bounded on $L^{p'}$ with norm that depends only on $p$. Then, the above last inequality gives us
\[
\int g\sum_{i\ge0}a_i\chi_{\Omega_i}\le 2 \|\sum_{i\ge0}a_i\chi_{\Omega_i\backslash\Omega_{i+1}}\|_{L^p}\|m(g)\|_{L^{p'}}\lesssim \big(\sum_{i\ge0} a_i^p|\Omega_{i}|\big)^{\frac{1}{p}} \|g\|_{L^{p'}},
\]
which implies the desired inequality in \eqref{thedesired}. To show why \( m \) is bounded on Lebesgue spaces, note that \( m \) is \( L^{\infty} \)-bounded with norm 1 and is of weak-\((1,1)\) type. These two facts, along with interpolation, imply that \( m \) is bounded on  \( L^{p'} \) for \( 1 < p < \infty \). The weak-\((1,1)\) bound for \( m \) follows from the fact that for \( \lambda > 0 \), we have
\[
|\{m(g) > \lambda\}| = |\bigcup_{\avr{|g|}{\Omega_i} > \lambda} \Omega_i| = |\Omega_{i_0}| \le \lambda^{-1} \|g\|_{L^1},
\]
where \( i_0 = \min\{i \ge 0 \mid \avr{|g|}{\Omega_i} > \lambda\} \).
This completes the proof.

\end{proof}
Another useful property of contracting families is that their large portions form a Carleson family of sets.
\begin{lemma}\label{largechunksparselemma}
	Let $\left\{\Omega_i\right\}_{i\ge0}$ be a contracting family and suppose $\left\{E_i\right\}_{i\ge0}$ is a family of measurable sets such that for $0<\eta\le1$, we have
	\[
	E_i\subset\Omega_{i}, \quad |E_i|\ge \eta |\Omega_{i}|, \quad i\ge0.
	\]
	Then, for any $A\subset \left\{0,1,2,\ldots\right\}$ we have
	\[
	\sum_{i\in A}|E_i|\le \frac{2}{\eta}|\bigcup_{i\in A}E_i|.
	\]
\end{lemma}
\begin{proof}
	Let $k=\min\{i\in A\}$ and note that
    \[
    \sum_{i\in A}|E_i|\le \sum_{i\in A}|\Omega_i|\le 2|\Omega_k|\le \frac{2}{\eta} |E_k|\le \frac{2}{\eta}|\bigcup_{i\in A}E_i|.
    \]
	
\end{proof}
For more on Carleson families of measurable sets, we refer the reader to \cite{MR3893778} and the references therein.

\subsection{Bi-parameter Dyadic Hardy Spaces}

Next, we turn to the definition of bi-parameter Hardy spaces in the dyadic setting. For a dyadic rectangle \(R\in\Dtwo\), we define
\[
h_R := h_I \otimes h_J, \quad \text{where } R = I \times J,
\]
with \( h_I \) and \( h_J \) being the $L^2$-normalized Haar wavelets associated with intervals \( I \) and \( J \), respectively.
 As it is well-known $\left\{h_R\right\}_{R\in\Dtwo}$ forms an orthonormal basis for $\ltwotwo$ and for $f\in\ltwotwo$ we have
\[
f=\sum_{R\in\Dtwo}f_Rh_R, \quad f_R:=\avr{f,h_R}{}.
\]	
To define dyadic Hardy spaces rigorously, we first make the following definition.

\begin{definition}
	 A dyadic distribution $f$ is a family of real numbers $\left\{f_R\right\}_{R\in\Dtwo}$, and formally is written as
	\[
	f:=\sum_{R\in \mathcal{D}\otimes\mathcal{D}}f_Rh_R.
	\]
\end{definition}	
For such an object the dyadic square function is defined as 
\[
S_d(f)(x):=\big(\sum_{R\in \mathcal{D}\otimes\mathcal{D}}f_R^2\frac{\chi_R}{|R|}(x)\big)^{\frac{1}{2}}.
\]	
\begin{definition}
For $0<p<\infty$, the space $\dyprodhp$ is the completion of the space of all real valued locally integrable functions $f$ with
\[
\dyprodhpn{f}:=\lprtwonorm{M_d(f)}<\infty.
\]
\end{definition}	
	
\begin{definition}
	For $0<p<\infty$, the space $\dotdyprodhp$ is the space of all dyadic distributions $f$ with
	\[
	\dotdyprodhpn{f}:=\lprtwonorm{S_d(f)}<\infty.
	\]
\end{definition}
Now, let us explain the relation between these two (quasi)-norms. When a priori $f$ is a bounded function with compact support, or more generally a function in $L_{loc}^1(\R^2)\cap\lqrtwo$ for some $0<q<\infty$, the two mentioned quantities are equivalent, with constants independent of this a priori information. The inequality 
\begin{equation}\label{mlessthans}
	\lprtwonorm{M_d(f)}\lesssim_p \lprtwonorm{S_d(f)},
\end{equation}
follows from sub-linearity of the operator $M_d$, and an atomic decomposition of $f$, which can be easily derived from the square function \cite{MR1320508,MR0584078}. The other direction 
\begin{equation}\label{slessthanm}
	\lprtwonorm{S_d(f)}\lesssim_p\lprtwonorm{M_d(f)},
\end{equation}
is harder to prove and follows from the following distributional inequality due to Brossard \cite{MR0602392}. See also \cite{MR1320508}.
\begin{theorem*}[Brossard]
There there exists a constant $C$ such that for any compactly supported function $f\in L^2(\R^2)$, and any $\delta>0$, we have 
	\begin{equation}\label{brossard}
		\int_{\left\{\left\{M_d(f)>\delta\right\}^{\sim}\right\}^c}S_d(f)^2\le C\left(\delta^2|\left\{M_d(f)>\delta\right\}|+\int_{\left\{M_d(f)\le\delta\right\}}M_d(f)^2\right).
	\end{equation}
\end{theorem*}
In \cite{MR0602392}, the inequality \eqref{brossard} was proved in the general setting of bi-parameter regular martingales. The analog of this inequality for bi-harmonic functions is due to Merryfield \cite{MR2611795}, and in the one-parameter setting, this inequality was previously established by Fefferman and Stein \cite{MR0447953}. To the best of our knowledge, the only application of this type of inequalities so far has been to prove \eqref{slessthanm}. However, in the next sections, we will show that the inequality \eqref{brossard} can be useful in certain constructions (see lemmas \ref{thetestfunction} and \ref{testfunctioncontinous}). It is also worth mentioning that all the arguments presented in the following sections extend verbatim to any number of parameters. We chose to work in the bi-parameter setting only because we could not find the analog of \eqref{brossard} for a higher number of parameters, though we believe such a theorem should hold \cite{MR2611795,MR0864369}.
\\

Here, we would like to mention that in the literature, the above two spaces, which are different, are both referred to as \(\dyprodhp\). This is also true for the one-parameter version \(H_d^p(\R)\) and the continuous versions \(H^p(\R)\), \(H^p(\R \otimes \R)\), etc. However, these spaces are not identical, and their equivalence must be understood through some a priori information or by using quotient norms. The reason lies in the cancellation within the square function, which is absent in the maximal function. For instance, since all bi-parameter Haar coefficients of a function of the form
\[
E(x_1, x_2) = f_1(x_1) + f_2(x_2)
\]
are zero, adding or subtracting such functions does not affect the bi-parameter dyadic square function but does change the dyadic maximal function.\\

The last function space to recall is the bi-parameter dyadic \(\text{BMO}\).
\begin{definition}
	The space $\dyprodbmo$ is the space of all dyadic distributions $f$ with
	\[
	\dyprodbmon{f}:=\sup_{\Omega} \big(|\Omega|^{-1}\sum_{R\subseteq\Omega}f_R^2\big)^{\frac{1}{2}}<\infty.
	\]
\end{definition}
In the above, the supremum is taken over all open subsets of the plane, and as Carleson famously showed \cite{carleson1974counter}, it is not sufficient to consider only rectangles, which is in sharp contrast with the one-parameter theory, where intervals or cubes can replace open sets.
	The next fact to recall is the bi-parameter John-Nirenberg inequality, which states that for any $0<p<\infty$, we have
\begin{equation}\label{johnnirenberg}
		\dyprodbmon{f}\simeq_p\sup_{\Omega} \avr{S_d(f|\Omega)^p}{\Omega}^{\frac{1}{p}}, \quad S_d(f|\Omega):=\big(\sum_{R\subseteq\Omega}f_R^2\frac{\chi_R}{|R|}\big)^\frac{1}{2},
\end{equation}
(see \cite{MR3052499} for a proof). Throughout the paper, we will use the notation on the right-hand side for localizations of the square function. Last but not least, we recall the well-known C. Fefferman's duality
\begin{equation}\label{honebmo}
	\dot{H}_d^1(\R \otimes \R)^* \cong \dyprodbmo,
\end{equation}
the proof of which in this setting is due to Bernard \cite{MR0539351}. We refer the reader to \cite{MR1320508} for the proofs and an exposition of Hardy spaces in the general setting of martingales.

\subsection{Bi-parameter Hardy spaces in the Continuous Setting} Let $\psi$ be a Schwartz function on $\R^2$ with 
\begin{align}
	&\text{supp}(\hat{\psi})\subseteq \{\xi=(\xi_1,\xi_2)\mid 0<\textbf{a}\le |\xi_1|,|\xi_2|\le \textbf{b}<\infty\},\label{adefinition}\\
	&\sum_{(j_1,j_2)\in\zed^2} \hat{\psi}(2^{-j_1}\xi_1,2^{-j_2}\xi_2)=1,\quad \xi_1,\xi_2\neq0.
\end{align}
Then, the bi-parameter Littlewood-Paley projections and the associated square function of a tempered distribution $f$ are defined as
\begin{align}
	&\Delta_j(f):=\psi_{2^{-j}}*f, \quad \hat{\psi}_{2^{-j}}(\xi):=\hat{\psi}(2^{-j_1}\xi_1,2^{-j_2}\xi_2), \quad j\in\zed^2,\\
	&S_\psi(f)(x):=\big(\sum_{j\in\zed^2}|\Delta_j(f)(x)|^2\big)^{\frac{1}{2}}, \quad x\in\R^2.
\end{align}
In addition, for a Schwartz function $\varphi$ with $\int\varphi=1$ let 
\[
M_\varphi(f)(x):=\sup_{t_1,t_2>0}|\varphi_t*f(x)|, \quad \varphi_t(x):= \frac{1}{t_1t_2}\varphi(\frac{x_1}{t_1},\frac{x_2}{t_2}), \quad x=(x_1,x_2)\in\R^2, \quad t_1,t_2>0.
\]
be the bi-parameter smooth vertical maximal operator and
\[
M_\varphi^*(f)(x):=\sup_{\substack{|x_1-y_1|\le t_1, |x_2-y_2|\le t_2\\t_1,t_2>0}}|\varphi_t*f(y_1,y_2)|,
\]
be its non-tangential analog.
\begin{definition}
For $0<p<\infty$, the space $H^p(\R\otimes \R)$ consists of all tempered distributions $f$ with
	\[
	 \|f\|_{H^p(\R\otimes \R)}:=\lprtwonorm{M_\varphi(f)}<\infty,
	\]
	and $\dot{H}^p(\R\times\R)$ is the space of all tempered distributions with
	\[
	 \|f\|_{\dot{H}^p(\R\otimes \R)}:=\lprtwonorm{S_\psi(f)}<\infty.
	\]
\end{definition}
In the above, both spaces are independent of the choice of $\varphi, \psi$, and $H^p(\R\otimes\R)$ is identical to $\lprtwo$ for $1<p<\infty$. Similar to the dyadic setting, with an a priori information such as $f\in\lqrtwo$ for some $0<q<\infty$, we have
\[
\|f\|_{H^p(\R\otimes \R)}\simeq_{p,\varphi,\psi}\|f\|_{\dot{H}^p(\R\otimes \R)}.
\]
See \cite{MR2611795} and \cite{MR0864369} for the proof of these.\\

Next, we recall quasi-orthogonal expansions and their properties. Let, $\theta$ be a Schwartz function whose Fourier transform is compactly supported and is equal to $1$ on the support of $\hat{\psi}$. Then we may write
\[
\Delta_j(g)=\theta_{2^{-j}}*\Delta_j(g)=\sum_{R \in \mathcal{D}\otimes\mathcal{D}_j}\int_{R}\theta_{2^{-j}}(x-y)\Delta_j(g)(y)dy, \quad j\in\zed^2,
\]
where in th above $\Dtwo_j$	denotes the collection of dyadic rectangles with sides $2^{-j_1}\times2^{-j_2}$ ($j=(j_1,j_2)$). Now, let
\begin{equation}\label{lambdardefinition}
	\lambda_R(g)=\sup_{y\in R}|\Delta_j(g)(y)||R|^{\frac{1}{2}}, \quad a_R(g)=\lambda_R(g)^{-1}\int_{R}\theta_{2^{-j}}(x-y)\Delta_j(g)(y)dy, \quad R\in \Dtwo_j.
\end{equation}	
Then, one can show that
\begin{equation}\label{equivalence}
	\|g\|_{\dot{H}^p(\R\otimes\R)}\simeq \dotdyprodhpn{\sum_{R \in \Dtwo}\lambda_R(g)h_R},\quad g=\sum_{R \in \Dtwo}\lambda_R(g) a_R(g),
\end{equation}	
roughly giving an isomorphism between $\dot{H}^p(\R\otimes\R)$ and $\dotdyprodhp$. In addition, for any sub-collection $\mathcal{C}$ of dyadic rectangles we have
\begin{equation}\label{hplesdyhp}
	\|\sum_{R \in \mathcal{C}}\lambda_R(g)a_R(g)\|_{\dot{H}^p(\R\otimes\R)}\lesssim	 \dotdyprodhpn{\sum_{R \in \mathcal{C}}\lambda_R(g)h_R}.
\end{equation}
The above inequalities follow from almost orthogonality of functions $a_R(g)$, Fefferman-Stein vector valued inequality, and a well-known inequality which captures the local constancy of band-limited functions. Below we bring its bi-parameter version.
\begin{theorem*}
	Let $f$ be a function with $\text{supp}(\hat{f})\subseteq \{\xi\mid |\xi_1|\le t_1, |\xi_2|\le t_2\}$, then we have 
	\begin{equation}\label{localconstancy}
		|f(y)|\lesssim_s(1+t_1t_2|x_1-y_1||x_2-y_2|) M_s(f)(x), \quad 0<s<\infty.
	\end{equation}
\end{theorem*}
See \cite{MR2463316}, p. 94, for the proof in the one-parameter setting. Here, we would like to mention that we could not find a proof of \eqref{equivalence} and \eqref{hplesdyhp} as stated above in the literature. However, the one-parameter arguments presented in \cite{MR2463316} work for any number of parameters with minor changes. See also \cite{MR0584078}.\\

At the end, let us recall the continuous $\text{BMO}$ in the product setting. Similar to the one-parameter theory, $\text{BMO}(\mathbb{R} \otimes \mathbb{R})$ can be defined in terms of Carleson measures, but we do not use this fact here and instead mention its quasi-orthogonal characterization, which is quite similar to \eqref{equivalence}. More precisely, $\text{BMO}(\mathbb{R} \otimes \mathbb{R})$ is the space of all tempered distributions $g$ such that the dyadic distribution
\[
\sum_{R \in \Dtwo}\lambda_R(g)h_R
\]
belongs to $\dyprodbmo$, and the equivalence of norms holds:
\[
\|g\|_{BMO(\mathbb{R} \otimes \mathbb{R})} \simeq \big(\sup_{\Omega}|\Omega|^{-1}\sum_{R\subseteq\Omega}\lambda_R(g)^2\big)^{\frac{1}{2}}.
\]
See \cite{MR0584078} for the proof.
	\section{Bi-parameter Dyadic Paraproducts}In the bi-parameter theory, there are different types of paraproducts arising in the product of two functions \cite{MR3052499}. The one considered here is of the form
	\[
	\pi_g(f) := \sum_{R \in \Dtwo} \langle f \rangle_R g_R h_R,
	\]
	where \( g \) is a dyadic distribution, \( f \in L_{\text{loc}}^1(\R^2) \), and \(\pi_g(f)\) is understood as a dyadic distribution.
	 It is easy to see that
	 \[
	 S_d(\pi_g(f)) \le S_d(g) M_d(f),
	 \]
	 and thus, by Hölder's inequality, we obtain
	 \[
	 \|\pi_g\|_{\dyprodhp \to \dotdyprodhq} \le \dotdyprodhrn{g}, \quad \frac{1}{q} = \frac{1}{p} + \frac{1}{r}, \quad 0 < p, q, r < \infty.
	 \]
Additionally, atomic decomposition together with the John-Nirenberg inequality imply that
\[
\|\pi_g\|_{\dyprodhp\to\dotdyprodhp}\lesssim \dyprodbmon{g}, \quad 0<p<\infty.
\]	
In this section we will show that in both cases the reverse direction holds, and this is the content of our main theorem.
\begin{theorem}\label{maindyadictheorem}
	Let $g$ be a dyadic distribution. Then we have
	\begin{align}
&\|\pi_g\|_{\dyprodhp\to\dotdyprodhq}\simeq\dotdyprodhrn{g}, \quad \frac{1}{q}=\frac{1}{p}+\frac{1}{r}, \quad 0<p,q,r<\infty,\tag{I}\\
&\|\pi_g\|_{\dyprodhp\to\dotdyprodhp}\simeq \dyprodbmon{g}, \quad 0<p<\infty.\tag{II}
	\end{align}
\end{theorem}
Using the duality relation \eqref{honebmo}, we get another characterization of $\dyprodbmo$, which is dual to (II).
\begin{corollary}
	Let $g$ be a function with only finitely many non-zero Haar coefficients. Then for the operator 
	\[
 \pi_g^*(f):=\sum_{R \in \Dtwo}f_Rg_R\frac{\chi_R}{|R|},
	\]
	we have
	\[
		\dyprodbmon{g}\simeq\|\pi_g^*\|_{\dyprodbmo\to\dyprodbmo}.
	\]
\end{corollary}
The proof of theorem \ref{maindyadictheorem} is based on the following two lemmas.
\begin{lemma}\label{sparsedominationforg}
Let $g$ be a dyadic distribution with Haar coefficients that are zero except for finitely many, and such that for an open subset $\Omega_0$ with finite measure, we have
\[
g=\sum_{R\subseteq\Omega_0}g_Rh_R.
\]
Then, there exists an absolute constant $0<\eta_0<1$ (independent of $g$ and $\Omega_0$) such that for any $0<\eta<\eta_0$, the following holds;\\

There exist a contracting family of open sets $\{\Omega_{i}\}_{i\ge0}$ starting from $\Omega_0$, and a sequence of numbers $\{\lambda_i\in\zed\cup\{-\infty\}\}_{i\ge0}$ such that
\begin{align}
&\eta|\Omega_{i}|\le |\{S_d(g|\Omega_{i})>2^{\lambda_i-1}\}|, \quad \text{if} \quad \lambda_i\in\zed \quad i\ge0,\label{thefirstpropertyforg}\\
&\dotdyprodhrn{g}\simeq_{\eta,r} \big(\sum_{i\ge0}2^{r\lambda_i}|\Omega_{i}|\big)^{\frac{1}{r}}, \quad 0<r<\infty.\label{thesecondpropertyforg}
\end{align}
In the above we use the convention that $2^{-\infty}=0$, and $r\cdot(-\infty)=-\infty$.
\end{lemma}	
\begin{lemma}\label{thetestfunction}
	For $0<p,\varepsilon\le1$ and any open set of finite positive measure $\Omega\subset\R^2$, there exists a function $\tilde{\chi}_{\Omega}\in \dyprodhp\cap \ltwotwo$ such that
\begin{align}
	&\dyprodhpn{\tilde{\chi}_{\Omega}}\lesssim_{p,\varepsilon} |\Omega|^{\frac{1}{p}},\label{thefirstpropertyfortest}\\
	&|\Omega'|\ge (1-\varepsilon)|\Omega|, \quad \Omega'=\big\{x\in\Omega\mid\quad\inf_{\substack{x\in R\\R\subseteq\Omega}}\avr{\tilde{\chi}_{\Omega}}{R}\ge \frac{1}{2}\big\}.\label{thesecondpropertyforf}	
\end{align}
\end{lemma}	
Let us accept these two lemmas and prove Theorem \ref{maindyadictheorem}.
\begin{proof}[Proof of Theorem \ref{maindyadictheorem}]
First, we consider the case (I), and to this aim let $$A=\|\pi_g\|_{\dyprodhp\to\dotdyprodhq}.$$ We must show that 
\begin{equation}\label{theoperatornorm}
	\dotdyprodhrn{g}\lesssim A.
\end{equation}
Now, observe that if $g'$ is a dyadic distribution with finitely many non-zero Haar coefficients which are the same as those of $g$, the operator norm of $\pi_{g'}$ is not larger than $A$. This follows because
\[
S_d(\pi_{g'}(f))\le S_d(\pi_g(f)).
\]
If we can show that \eqref{theoperatornorm} holds with \( g \) replaced by \( g' \), then by the monotone convergence theorem, we can easily deduce \eqref{theoperatornorm} for \( g \). Therefore, without loss of generality, we assume that \( g \) has only finitely many non-zero coefficients, and that the associated dyadic rectangles are all contained in \(\Omega_0\).
We then apply Lemma \ref{sparsedominationforg} to \( g \) with sufficiently small \(\eta\), which yields a finite contracting sequence of open sets \(\{\Omega_i\}_{i \ge 0}\) and numbers \(\{\lambda_i\}_{i \ge 0}\) with the described properties. Next, when \(0 < p \le 1\), we apply Lemma \ref{thetestfunction} with \(\varepsilon = \frac{1}{2}\eta\) to each \(\Omega_i\) to obtain the function \(\tilde{\chi}_{\Omega_i}\). For \(1 < p < \infty\), we set \(\tilde{\chi}_{\Omega_i} = \chi_{\Omega_i}\) for \(i \ge 0\).\\

As the next step, we take a sequence of independent Bernoulli random variables \(\{\omega_i\}_{i \ge 0}\) with \(\mathbb{P}(\omega_i = \pm 1) = \frac{1}{2}\) and construct the following random function:
\[
f_\omega = \sum_{i \ge 0} \omega_i 2^{t \lambda_i} \tilde{\chi}_{\Omega_i}, \quad \text{where } t = \frac{r}{q} - 1.
\]
Now, from \eqref{mlessthans}, we have
\[
\int \big| \sum_{R \in \Dtwo} g_R \avr{f_\omega}{R} h_R \big|^q \lesssim A^q \dyprodhpn{f_\omega}^q.
\]
Additionally, if we choose another sequence of independent Bernoulli random variables \(\{\epsilon_{R}\}_{R \in \Dtwo}\) and multiply these signs to the coefficients of \(g\), the operator norm of the resulting dyadic paraproduct remains unchanged. This yields
\[
\int\big|\sum_{R,i\ge0}\epsilon_{R}\omega_ig_R2^{t\lambda_i}\avr{\tilde{\chi}_{\Omega_{i}}}{R}h_R\big|^q=\int\big|\sum_{R\in\Dtwo}\epsilon_{R}g_R\avr{f_\omega}{R}h_R\big|^q\lesssim A^q\dyprodhpn{f_\omega}^q.
\]
Now, taking the expectation with respect ot both variables and using the bi-parameter version of Khintchine inequality gives us
\[
\int\big(\sum_{R,i\ge0}g_R^22^{2t\lambda_i}\avr{\tilde{\chi}_{\Omega_{i}}}{R}^2\frac{\chi_R}{|R|}\big)^\frac{q}{2}\lesssim A^q\mathbb{E}\dyprodhpn{f_\omega}^q.
\]
Next, call the function under the sign of integral $F$ and rewrite the above inequality as
\begin{equation}\label{bigfequation}
	\|F\|_{L^1(\R^2)}\lesssim A^q\mathbb{E}\dyprodhpn{f_\omega}^q, \quad F=\big(\sum_{R,i\ge0}g_R^22^{2t\lambda_i}\avr{\tilde{\chi}_{\Omega_{i}}}{R}^2\frac{\chi_R}{|R|}\big)^\frac{q}{2}.
\end{equation}
At this point, let us first estimate the right hand side of the above inequality, and assume initially that $0<p\le1$. From sub-additivity, \eqref{thefirstpropertyfortest} in Lemma \ref{thetestfunction}, and the fact that $tp=r$ we obtain
\[
\dyprodhpn{f_\omega}^p\le\sum_{i\ge0}2^{pt\lambda_i}\dyprodhpn{\tilde{\chi}_{\Omega_{i}}}^p\lesssim\sum_{i\ge0}2^{pt\lambda_i}|\Omega_{i}|=\sum_{i\ge0}2^{r\lambda_i}|\Omega_{i}|.
\]
In the case that $1<p<\infty$ we also have
\[
\dyprodhpn{f_\omega}\simeq \lprtwonorm{f_\omega}\le \lprtwonorm{\sum_{i\ge0}2^{t\lambda_i}\chi_{\Omega_i}}\simeq \big(\sum_{i\ge0}2^{r\lambda_i}|\Omega_{i}|\big)^{\frac{1}{p}},
\]
where in the above we used Lemma \ref{lpsparselemma}. Therefore, from the above and \eqref{bigfequation} we must have
\begin{equation}\label{thelonenormofF}
	\|F\|_{L^1(\R^2)}\lesssim A^q\big(\sum_{i\ge0}2^{r\lambda_i}|\Omega_{i}|\big)^{\frac{q}{p}}.
\end{equation}
The next step is to observe that \eqref{bigfequation} implies
\[
F\ge 2^{tq\lambda_i}\big(\sum_{R\subseteq\Omega_i}g_R^2\avr{\tilde{\chi}_{\Omega_{i}}}{R}^2\frac{\chi_R}{|R|}\big)^{\frac{q}{2}}.
\]
Then, since from \eqref{thesecondpropertyforf} we have
\[
\avr{\tilde{\chi}_{\Omega_{i}}}{R}\ge\frac{1}{2}, \quad R\subseteq\Omega_{i}, \quad x\in R, \quad x\in \Omega_{i}',
\]
we conclude that
\begin{equation}\label{Fislarge}
	F(x)\ge 2^{(1+t)q\lambda_i-2q}=2^{r\lambda_i-2q}, \quad x\in E_i= \Omega_i'\cap\left\{S_d(g|\Omega_{i})>2^{\lambda_i-1}\right\}, \quad i\ge0.
\end{equation}
Now, since
\[
|\Omega_{i}'|\ge(1-\frac{1}{2}\eta)|\Omega_{i}|, \quad |\left\{S_d(g|\Omega_{i})>2^{\lambda_i-1}\right\}|\ge \eta|\Omega_{i}|, \quad i\ge0,
\]
we have that
\[
|E_i|\ge\frac{1}{2}\eta|\Omega_{i}|, \quad E_i\subseteq \Omega_{i}, \quad i\ge0,
\]
which together with \eqref{Fislarge} and Lemma \ref{largechunksparselemma} implies that
\begin{align}
		&\sum_{i\ge0} 2^{r\lambda_i}|\Omega_{i}|=\sum_{k\in\zed}2^{rk}\sum_{\lambda_i=k}|\Omega_{i}|\lesssim_{\eta}\sum_{k\in\zed}2^{rk}\sum_{\lambda_i=k}|E_{i}|
		\lesssim\label{d}\\ &\sum_{k\in\zed}2^{rk}|\bigcup_{\lambda_i=k}E_i|\le\sum_{k\in\zed}2^{rk}|\{F>2^{rk-2q}\}|\simeq \|F\|_{L^1(\R^2)}.\label{dddd}
\end{align}
From this and \eqref{thelonenormofF} we obtain
\begin{equation*}
	\sum_{i\ge0} 2^{r\lambda_i}|\Omega_{i}|\lesssim  A^q\big(\sum_{i\ge0}2^{r\lambda_i}|\Omega_{i}|\big)^{\frac{q}{p}},
\end{equation*}
which after noting that the sum appearing on both sides is finite gives us
\begin{equation*}\label{ddd}
	\big(\sum_{i\ge0} 2^{r\lambda_i}|\Omega_{i}|\big)^{\frac{1}{r}}\lesssim A.
\end{equation*}
Finally, we recall \eqref{thesecondpropertyforg} and obtain
\[
\dotdyprodhrn{g}\lesssim A,
\]
which is the desired inequality in \eqref{theoperatornorm}, and this finishes the proof of case (I).\\

Now we turn to the case (II), which can be proven by almost the same argument, and we outline only the required changes. This time, we need to show that for any open subset of finite measure $\Omega_0$ we have
\[
\int S_d(g|\Omega_0)^p\lesssim A^p |\Omega_0|, \quad A=\|\pi_g\|_{\dyprodhp\to\dotdyprodhp}.
\]
We note that if \(1 < p < \infty\), the result follows trivially by inserting the function \(\chi_{\Omega_0}\) into the operator and observing that
\[
S_d(g|\Omega_0) \le S_d(\pi_g(\chi_{\Omega_0})).
\]
Thus, we only need to consider the case \(0 < p \le 1\). In this scenario, the only change required in the argument is to set \(t = 0\) and note that
\[
\dyprodhpn{f_\omega}^p \lesssim |\Omega_0|.
\]
Then, the inequality \eqref{thelonenormofF} is replaced with
\[
\|F\|_{L^1(\R^2)} \lesssim A^p |\Omega_0|,
\]
and in \eqref{d} and \eqref{dddd} we replace \(r\) with \(p\), yielding
\[
\int S_d(g|\Omega_0)^p \simeq \sum_{i \ge 0} 2^{p \lambda_i} |\Omega_i| \lesssim \|F\|_{L^1(\R^2)} \lesssim A^p |\Omega_0|.
\]
This establishes the result we sought, thereby completing the proof of case (II) and the proof of Theorem \ref{maindyadictheorem}.

\end{proof}
Now, we turn to the proof of Lemma \ref{sparsedominationforg}, which is based on an iteration of a well-known argument in multi-parameter theory.
\begin{proof}[Proof of Lemma \ref{sparsedominationforg}]
Let \(\lambda_0\) be the smallest number in \(\mathbb{Z} \cup \{-\infty\}\) such that
\begin{equation}
	|\{S_d(g|\Omega_0) > 2^{\lambda_0}\}| \le \eta |\Omega_0|.
\end{equation}
The above inequality is satisfied for sufficiently large values of \(\lambda_0 \in \mathbb{Z}\), ensuring the existence of such a number. Also, if \(\lambda_0 \in \mathbb{Z}\), we must have
\[
\eta|\Omega_0| \le |\{S_d(g|\Omega_0) > 2^{\lambda_0 - 1}\}|.
\]
Next, let 
\[
\Omega_1 := \{S_d(g|\Omega_0) > 2^{\lambda_0}\}^{\sim} \cap \Omega_0,
\]
and note that if \(\eta\) is small enough, we have
\[
|\Omega_1| \lesssim \eta |\Omega_0| \le \frac{1}{2} |\Omega_0|.
\]
Now, we repeat the process by replacing \(\Omega_0\) with \(\Omega_1\), and get \(\lambda_1\) and \(\Omega_2\). Continuing this, we end up with a sequence of contracting open sets \(\{\Omega_{i}\}_{i \ge 0}\) and numbers \(\lambda_i \in \mathbb{Z} \cup \{-\infty\}\) such that
\begin{align}
	&\Omega_{i+1}=\{S_d(g|\Omega_i)>2^{\lambda_i}\}^{\sim}\cap\Omega_i, \quad i\ge0,\\
	&\eta|\Omega_{i}|\le|\{S_d(g|\Omega_i)>2^{\lambda_i-1}\}|, \quad \text{if} \quad \lambda_i\in\zed.\label{goodchunk}
\end{align}
 Clearly these sets and numbers satisfy \eqref{thefirstpropertyforg}, and it remains to show that \eqref{thesecondpropertyforg} holds as well. First, we show that
 \[
 \sum_{i \ge 0} 2^{r\lambda_i} |\Omega_{i}| \lesssim \|S_d(g)\|_{L^r(\mathbb{R}^2)}^r,
 \]
 which follows from
 \[
 \sum_{i \ge 0} 2^{r\lambda_i} |\Omega_{i}| = \sum_{k \in \mathbb{Z}} \sum_{\lambda_i = k} 2^{rk} |\Omega_{i}| \le \eta^{-1} \sum_{k \in \mathbb{Z}} 2^{rk} \sum_{\lambda_i = k} |\{S_d(g|\Omega_i) > 2^{k-1}\}|,
 \]
 and noting that from \eqref{goodchunk}, we are allowed to apply Lemma \ref{largechunksparselemma} to get
 \[
 \sum_{i \ge 0} 2^{r\lambda_i} |\Omega_{i}| \lesssim \sum_{k \in \mathbb{Z}} 2^{rk} |\{S_d(g) > 2^{k-1}\}| \simeq \|S_d(g)\|_{L^r(\mathbb{R}^2)}^r.
 \]
 In order to get the other direction, we decompose $g$ as
\begin{equation}\label{atomic decomposition}
	g=\sum_{i\ge0}g_i, \quad g_i:=\sum_{\substack{R\subseteq\Omega_{i}\\R\nsubseteq\Omega_{i+1}}}g_Rh_R, \quad i\ge0,
\end{equation}
with the convention that if $\Omega_{i}=\emptyset$, then $g_i=0$. Then, we consider two separate cases either $0<r\le2$ or $2<r<\infty$. For the first case, we apply sub-linearity and get
\[
\lrrtwonorm{S_d(g)}^r=\int \big|\sum_{i\ge0}S_d(g_i)^2\big|^{\frac{r}{2}}\le\sum_{i\ge0}\int_{\Omega_{i}} \big|S_d(g_i)^2\big|^{\frac{r}{2}}\le\sum_{i\ge0}|\Omega_{i}|^{1-\frac{r}{2}}\left(\int S_d(g_i)^2\right)^{\frac{r}{2}},
\]
where in the last estimate we used H\"older's inequality. Therefore, it is enough to show that
\begin{equation}\label{lonenormofs2}
	\int S_d(g_i)^2 \lesssim 2^{2\lambda_i} |\Omega_{i}|, \quad i \ge 0,
\end{equation}
which together with the previous estimate gives us
\[
\|S_d(g)\|_{L^r(\mathbb{R}^2)}^r \lesssim \sum_{i \ge 0} 2^{r\lambda_i} |\Omega_{i}|,
\]
which is what we are looking for. Now, to see why \eqref{lonenormofs2} holds, note that
\[
\int S_d(g_i)^2 = \sum_{\substack{R \subseteq \Omega_{i} \\ R \nsubseteq \Omega_{i+1}}} g_R^2 \le 2 \sum_{\substack{R \subseteq \Omega_{i} \\ R \nsubseteq \Omega_{i+1}}} g_R^2 \left| \{S_d(g|\Omega_{i}) > 2^{\lambda_i}\}^c \cap R \right| |R|^{-1},
\]
where the last inequality follows from the fact that \( R \nsubseteq \Omega_{i+1} \), and this gives us
\[
\int S_d(g_i)^2 \le 2 \int_{\{S_d(g|\Omega_{i}) > 2^{\lambda_i}\}^c} S_d(g|\Omega_{i})^2 \le 2 \cdot 2^{2\lambda_i} |\Omega_{i}|, \quad i \ge 0,
\]
showing that \eqref{lonenormofs2} holds, and the proof of this case is finished. Next, we consider the case \(2 < r < \infty\), which follows from duality together with a similar argument to the one presented above. Let \(\varphi\) be a function with \(\|\varphi\|_{L^{\left(\frac{r}{2}\right)'}(\mathbb{R}^2)} = 1\).
 Then we have
\begin{align*}
	&\int S_d(g)^2\varphi=\sum_{i\ge0}\int S_d(g_i)^2\varphi =\sum_{i\ge0}\sum_{\substack{R\subseteq\Omega_{i}\\R\nsubseteq\Omega_{i+1}}}g_R^2\avr{\varphi}{R}\le\\
	&2\sum_{i\ge0}\sum_{\substack{R\subseteq\Omega_{i}\\R\nsubseteq\Omega_{i+1}}}g_R^2|\{S_d(g|\Omega_{i})>2^{\lambda_i}\}^c\cap R||R|^{-1}|\avr{\varphi}{R}|\le\\
	&2\sum_{i\ge0} \int_{\{S_d(g|\Omega_{i})>2^{\lambda_i}\}^c}S_d(g|\Omega_{i})^2M_d(\varphi)\le2\int \sum_{i\ge0}2^{2\lambda_i}\chi_{\Omega_{i}}M_d(\varphi)\le\\
	&2\|\sum_{i\ge0}2^{2\lambda_i}\chi_{\Omega_{i}}\|_{L^{\frac{r}{2}}(\R^2)}\|M_d(\varphi)\|_{L^{(\frac{r}{2})'}(\R^2)}\lesssim \left(\sum_{i\ge0}2^{r\lambda_i}|\Omega_{i}|\right)^{\frac{2}{r}},
\end{align*}
where for the last estimate Lemma \ref{lpsparselemma} is used, and this shows that
\[
\lrrtwonorm{S_d(g)}\lesssim\big(\sum_{i\ge0}2^{r\lambda_i}|\Omega_{i}|\big)^{\frac{1}{r}},
\]
which completes the proof of this case and Lemma \ref{sparsedominationforg}.
\end{proof}
\begin{remark}
The decomposition \eqref{atomic decomposition}
\[
g = \sum_{i \ge 0} g_i, \quad g_i := \sum_{\substack{R \subseteq \Omega_i \\ R \nsubseteq \Omega_{i+1}}} g_R h_R, \quad i \ge 0,
\]
is an atomic decomposition of \( g \) with the property that the supporting open sets of its atoms form a contracting family. Indeed, the above argument shows that \( g_i \) are \( L^p \)-atoms (although not normalized) for \(\dyprodhr\) for any positive values of \( p \) and \( r \). See \cite{MR0671315} for the counterpart of this result in the one-parameter theory. See also \cite{MR2157745} p. 42.

\end{remark}	
As our last job in this section we give the proof of Lemma \ref{thetestfunction}.
\begin{proof}[Proof of Lemma \ref{thetestfunction}]
	Let $0<\delta<1$ be a small number to be determined later, and consider the following three enlargements of $\Omega$
	\[
	\Omega_1:=\{M_d(\chi_{\Omega})>\delta\}, \quad \Omega_2:=\tilde{\Omega}_1, \quad \Omega_3:=\{M_d(\chi_{\Omega_2})>\delta\},
	\]
	and let
	\[
\tilde{\chi}_\Omega=\chi_{\Omega}-f,\quad	f=\sum_{R\nsubseteq\Omega_3}\avr{\chi_{\Omega},h_R}{}h_R.
	\]
	Then since
	\[
	\tilde{\chi}_\Omega=\sum_{R\subseteq\Omega_3}\avr{\chi_{\Omega},h_R}{}h_R,
	\]
	from H\"older's inequality and the fact that $|\Omega_3|\lesssim_{\delta}|\Omega|$ we have
	\[
	\int S_d(\tilde{\chi}_\Omega)^p\le |\Omega_3|^{1-\frac{p}{2}}\left(\int S_d(\chi_{\Omega})^2\right)^{\frac{p}{2}}\lesssim_{\delta,p}|\Omega|.
	\]
So, $\tilde{\chi}_\Omega$ satisfies the first property in \eqref{thefirstpropertyfortest}, and it remains to choose $\delta$ so small that the second property holds as well. To this aim, we note that
\[
\avr{\tilde{\chi}_\Omega}{R}+\avr{f}{R}=1, \quad R\subseteq \Omega,
\]
and thus
\[
\Omega'=\big\{x\in\Omega|\quad\inf_{\substack{x\in R\\R\subseteq\Omega}}\avr{\tilde{\chi}_{\Omega}}{R}\ge \frac{1}{2}\big\}\supseteq\{M_d(f)\ge\frac{1}{2}\}^c\cap\Omega,
\]
which implies that if $f$ is such that we have
\begin{equation}\label{badset}
	|\{M_d(f)\ge\frac{1}{2}\}|\le \varepsilon|\Omega|,
\end{equation}
then $\tilde{\chi}_{\Omega}$ satisfies \eqref{thesecondpropertyforf} and we are done. Now since we have
\begin{equation}\label{ifltoissmal}
	|\{M_d(f)\ge\frac{1}{2}\}|\lesssim \int_{\R^2} |f|^2,
\end{equation}
it is enough to show that the right hand is small. Here, we observe that
\begin{equation}\label{decompositionforf}
	\int_{\R^2} |f|^2=\int_{\R^2} S_d(f)^2=\int_{\Omega_2} S_d(f)^2+\int_{\Omega_2^c} S_d(f)^2,
\end{equation}
and may estimate the first term by
\begin{equation}\label{thesecondpart}
	\int_{\Omega_2} S_d(f)^2=\sum_{R\nsubseteq\Omega_3}\avr{\chi_{\Omega},h_R}{}^2 \frac{|\Omega_2\cap R|}{|R|}\le \delta \sum_{R\nsubseteq\Omega_3}\avr{\chi_{\Omega},h_R}{}^2\le \delta|\Omega|,
\end{equation}
where in the above we used the fact that $R\nsubseteq\Omega_3$. For the second term we have
\[
\int_{\Omega_2^c} S_d(f)^2\le\int_{\Omega_2^c} S_d(\chi_{\Omega})^2\lesssim \delta^2|\{M_d(\chi_{\Omega})>\delta\}|+\int_{\left\{M_d(\chi_{\Omega})\le\delta\right\}}M_d(\chi_{\Omega})^2,
\]
which follows from \eqref{brossard}. Now, using boundedness of $M_d$ of $L^{\frac{3}{2}}(\R^2)$ yields
\begin{equation}\label{firstpart}
	\int_{\Omega_2^c} S_d(f)^2\lesssim\delta^2|\{M_d(\chi_{\Omega})>\delta\}|+\delta^\frac{1}{2}\int_{\left\{M_d(\chi_{\Omega})\le\delta\right\}}M_d(\chi_{\Omega})^{\frac{3}{2}}\lesssim \delta^\frac{1}{2}|\Omega|,
\end{equation}
and putting \eqref{ifltoissmal},\eqref{decompositionforf}, \eqref{thesecondpart}, and \eqref{firstpart} together gives us
\[
|\{M_d(f)\ge\frac{1}{2}\}|\lesssim\int_{\R^2} |f|^2\lesssim \delta^\frac{1}{2}|\Omega|,
\]
showing that by choosing $\delta$ small enough \eqref{badset} holds, and this completes the proof.
\end{proof}	

Here, it is worth mentioning that in the one-parameter theory the same construction yields a function with the stronger property that
	\[
	\tilde{\chi}_{\Omega}(x)\ge \frac{1}{2}, \quad x\in\Omega.
	\]
To see this, let $f=\sum_{I\nsubseteq\tilde{\Omega}}f_Ih_I$ and $\tilde{\chi}_\Omega=\chi_{\Omega}-f$. Then, the function $f$, is constant on each maximal dyadic interval of $\tilde{\Omega}$, and thus is not larger than $\frac{1}{2}$ on $\Omega$, which proves the above inequality. Regarding this, we ask the following question:\\

\textbf{Question.} Let $0<p\le1$, and $\Omega$ be an open subset with $|\Omega|<\infty.$ Does there exist a function $\tilde{\chi}_\Omega$ with the following two properties?
\[
\dyprodhpn{\tilde{\chi}_\Omega}\lesssim_p |\Omega|^{\frac{1}{p}}, \quad \tilde{\chi}_\Omega(x)\ge \frac{1}{2}, \quad x\in\Omega.
\]

\section{Bi-Parameter Fourier Paraproducts}	In this section we explain how similar results can be obtained for Fourier paraproducts of the form
\[
\Pi_g(f)(x):=\sum_{j\in\zed^2}\varphi_{2^{-j}}*f(x)\Delta_j(g)(x), \quad x\in\R^2,
\]
where in the above $\hat{\varphi}\subseteq\{\xi\mid|\xi_1|,|\xi_2|\le \textbf{a}'\}$ with $\textbf{a}'<\textbf{a}$, and $\textbf{a}$ is the same as in \eqref{adefinition}. Then, one can show that
\begin{align*}
	&\|\Pi_g\|_{H^p(\R\otimes\R)\to\dot{H}^q(\R\otimes\R)}\lesssim \|g\|_{\dot{H}^r(\R\otimes\R)}, \quad \frac{1}{q}=\frac{1}{p}+\frac{1}{r}, \quad 0<p,q,r<\infty,\\
	&\|\Pi_g\|_{H^p(\R\otimes\R)\to\dot{H}^p(\R\otimes\R)}\lesssim \|g\|_{BMO(\R\otimes\R)}, \quad 0<p<\infty.
\end{align*}
Indeed, the above inequalities follow from the support properties of $\hat{\varphi}$ and $\hat{\psi}$ and
\begin{align}
		&\|\mathcal{S}_g(f)\|_{\lqrtwo}\lesssim\|g\|_{\dot{H}^r(\R^2)} \|f\|_{H^p(\R\otimes\R)},\quad \frac{1}{q}=\frac{1}{p}+\frac{1}{r}, \quad 0<p,q,r<\infty,\label{boundforsg}\\
		&\|\mathcal{S}_g(f)\|_{\lprtwo}\lesssim \|g\|_{BMO(\R\otimes\R)}\|f\|_{H^p(\R\otimes\R)}, \quad 0<p<\infty,\label{lplp}
\end{align}
where in the above
\[
\mathcal{S}_g(f)(x):=\big(\sum_{j\in\zed^2}|\varphi_{2^{-j}}*f(x)\Delta_j(g)(x)|^2\big)^\frac{1}{2}, \quad x\in\R^2,
\]
\cite{MR3052499}.
Here, we show that the converse of \eqref{boundforsg} and \eqref{lplp} holds.
\begin{theorem}\label{fouriertheorem}
	Let g be a tempered distribution, and $\mathcal{S}_g$ be as above. Then,
	\begin{align}
		&\|\mathcal{S}_g\|_{H^p(\R\otimes\R)\to L^q(\R^2)}\simeq \|g\|_{\dot{H}^r(\R\otimes\R)}, \quad \frac{1}{q}=\frac{1}{p}+\frac{1}{r}, \quad 0<p,q,r<\infty,\tag{i}\\
		&\|\mathcal{S}_g\|_{H^p(\R\otimes\R)\to L^p(\R^2)}\simeq \|g\|_{BMO(\R\otimes\R)}, \quad \quad 0<p<\infty.\tag{ii}
	\end{align}
\end{theorem}	

\begin{proof}[Proof of Theorem]
We prove only case (i), since the other case follows with exactly the same argument. Let, $A=\|\mathcal{S}_g\|_{H^p(\R\otimes\R)\to L^q(\R^2)},$
we need to show that 
\[
\|g\|_{\dot{H}^r(\R\otimes\R)}\lesssim A.
\]
To this aim, recall the notation introduced in \eqref{lambdardefinition}, let $x_R\in R$ be such that
\[
|\Delta_jg(x_R)|=\sup_{y\in R}|\Delta_j(g)(y)|, \quad R\in\Dtwo_j, \quad j\in\zed^2,
\]
and define
\[
\pi'(f):=\sum_{j\in\zed^2}\sum_{R \in \Dtwo_j}\varphi_{2^{-j}}*f(x_R)\lambda_R(g)h_R.
\]
Then we note that the function $\varphi_{2^{-j}}*f\Delta_j(g)$, has Fourier support in $\{\xi\mid |\xi_1|\lesssim 2^{j_1},|\xi_2|\lesssim 2^{j_2}\}$, and thus we may apply \eqref{localconstancy} and get
\[
|\varphi_{2^{-j}}*f(x_R)|\sup_{y\in R}|\Delta_j(g)(y)|\le\sup_{y\in R}|\varphi_{2^{-j}}*f(y)\Delta_j(g)(y)|\lesssim M_s(\varphi_{2^{-j}}*f\Delta_j(g))(x), \quad x\in R,
\]
which implies that
\[
\sum_{j\in\zed^2}\sum_{R \in \Dtwo}|\varphi_{2^{-j}}*f(x_R)\lambda_R(g)|^2\frac{\chi_R}{|R|}\lesssim \sum_{j\in\zed^2}M_s(\varphi_{2^{-j}}*f\Delta_j(g))^2(x).
\]
Then applying Fefferman-Stein inequality with $0<s<q$ yields 
\[
\dotdyprodhqn{\pi'(f)}\lesssim \|\big(\sum_{j\in\zed^2}M_s(\varphi_{2^{-j}}*f\Delta_j(g))^2\big)^\frac{1}{2}\|_{\lqrtwo}\lesssim \|\mathcal{S}_g(f)\|_{\lqrtwo}\le A \|f\|_{H^p(\R\otimes \R)}.
\]
So we have
\[
\|\pi'\|_{H^p(\R\otimes\R)\to\dotdyprodhq}\lesssim A.
\]
Now, it follows from \eqref{equivalence} that our task is the show that 
\[
\dotdyprodhqn{\sum_{R\in\Dtwo}\lambda_R(g)h_R} \lesssim A,
\]
which follows from an argument identical to the one presented in the proof of Theorem \ref{maindyadictheorem} if we replace $\tilde{\chi}_\Omega$ with its counterpart $\tilde{\tilde{\chi}}_\Omega$ in the following lemma, and this completes the proof.

\end{proof}	
	
\begin{lemma}\label{testfunctioncontinous}
	Let $0<p,\varepsilon\le1$, and $\varphi$ be a Schwartz function with $\int\varphi=1$. Then for any open set $\Omega$ with $|\Omega|<\infty$, there exists a function $\tilde{\tilde{\chi}}_\Omega\in H^p(\R\otimes\R)\cap \ltwotwo$ such that
	\begin{align}
		& \|\tilde{\tilde{\chi}}_\Omega\|_{H^p(\R\otimes\R)}\lesssim_{p,\varepsilon} |\Omega|^{\frac{1}{p}},\label{11111}\\
		& |\Omega'|\ge(1-\varepsilon)|\Omega|, \quad \Omega'=\big\{x\in\Omega\mid \inf_{\substack{x,y\in R\\ R\subseteq\Omega, R\in \Dtwo_j, j\in \zed^2}}\varphi_{2^{-j}}*\tilde{\tilde{\chi}}_\Omega(y)\ge\frac{1}{2}\big\}.\label{2222}
	\end{align}
\end{lemma}	
\begin{proof}
	First, we note that since $\varphi$ decays rapidly and $\int\varphi=1$, we may choose a large constant $\alpha$, depending only on $\varphi$ such that
	\[
	|\varphi|_{2^{-j}}*\chi_{(\alpha R)^c}(y)<\frac{1}{4}, \quad y\in R, \quad R\in\Dtwo_j, \quad j\in\zed^2,
	\]
where $\alpha R$ is the concentric dilation of $R$ with $\alpha$. Thus, if $O$ is an open subset with the property that $\alpha R\subseteq O$ whenever $R\subseteq \Omega$ we must have
\begin{equation}\label{3333}
		\varphi_{2^{-j}}*\chi_O(y)>\frac{3}{4}, \quad y\in R\subseteq \Omega, \quad R\in\Dtwo_j, \quad j\in\zed^2.
\end{equation}
	So, let 
	\[
	O=\{M(\chi_{\Omega})\ge \alpha^{-1}\},
	\]
	be the first enlargement of $\Omega$, where $M$ is the strong maximal operator. Then, let 
	\[
	\chi_{O}=\sum_{R \in \Dtwo}\lambda_R(\chi_{O})a_R(\chi_{O}),
	\]
	be the quasi-orthogonal expansion of $\chi_{O}$ as in \eqref{equivalence}, and define the function
	\[
	E:=\sum_{R \in \Dtwo}\lambda_R(\chi_{O})h_R.
	\]
	Now, we follow the same strategy as in Lemma \ref{thetestfunction} as it follows. Take a small number $0<\delta<1$, to be determined later and set 
	\[
	O_1:=\{M_d(E)>\delta\}, \quad O_2:=\tilde{O}_1, \quad O_3:=\{M_d(\chi_{O_2})>\delta\},
	\]
	and decompose $E$ and $\chi_{O}$ as
\begin{align}
	&E=\sum_{R\subseteq O_3}E_Rh_R+\sum_{R\nsubseteq O_3}E_Rh_R=\tilde{E}+\tilde{f}\\
	&\chi_{O}=\sum_{R\subseteq O_3}\lambda_R(\chi_{O})a_R(\chi_{O})+\sum_{R\nsubseteq O_3}\lambda_R(\chi_{O})a_R(\chi_{O})=\tilde{\tilde{\chi}}_\Omega+f.
\end{align}
	Next, we note that \eqref{hplesdyhp}, H\"older's inequality, and \eqref{equivalence} imply that
\begin{equation}
\|\tilde{\tilde{\chi}}_\Omega\|_{H^p(\R\otimes\R)}^p\lesssim	\dotdyprodhpn{\tilde{E}}^p\le\int_{O_3}S_d(E|O_3)^p\le|O_3|^{1-\frac{p}{2}}\left
(\int S_d(E)^2\right)^{\frac{p}{2}}\lesssim_{\delta,p}|O|\lesssim_{\delta,p}|\Omega|.
\end{equation}
Therefore, $\tilde{\tilde{\chi}}_\Omega$ satisfies \eqref{11111}, and it remains to show that \eqref{2222} holds as well. Now, similar to Lemma \ref{thetestfunction} we may estimate the $L^2$-norm of $f$ as
\begin{align*}
		&\int_{\R^2}|f|^2\lesssim\int_{\R^2}|\tilde{f}|^2=\int_{O_2}|S_d(\tilde{f})|^2+\int_{O_2^c}|S_d(\tilde{f})|^2\le \sum_{R\nsubseteq O_3}E_R^2\frac{|R\cap O_2|}{|R|} + \int_{O_2^c}|S_d(E)|^2\le\\
		& \delta \sum_{R\nsubseteq O_3}E_R^2 +\int_{\{\{M_d(E)>\delta\}^{\sim}\}^c}|S_d(E)|^2\lesssim \delta \int_{\R^2}|E|^2+ \delta^2 |\{M_d(E)>\delta\}|+\\
		&\int_{M_d(E)\le\delta}M_d(E)^2\lesssim
		 \delta \int_{\R^2}|E|^2 + \delta^{\frac{1}{2}} \int_{\R^2}|E|^{\frac{3}{2}}\lesssim \delta^{\frac{1}{2}}|O|\lesssim \delta^{\frac{1}{2}}|\Omega|.
\end{align*}
In the above we used \eqref{hplesdyhp} and \eqref{brossard}. Next, recall \eqref{3333} and note that
\[
	\varphi_{2^{-j}}*\tilde{\tilde{\chi}}_\Omega(y)+\varphi_{2^{-j}}*f(y)>\frac{3}{4}, \quad y\in R\subseteq \Omega, \quad R\in\Dtwo_j, \quad j\in\zed^2,
\]
which implies that
\[
\Omega'\supseteq \big\{x\mid \sup_{\substack{x,y\in R\\ R\subseteq\Omega, R\in \Dtwo_j, j\in \zed^2}}\varphi_{2^{-j}}*f(y)>\frac{1}{4} \big\}^c\cap\Omega\supseteq\big\{M_\varphi^*(f)(x)>\frac{1}{4}\big\}^c\cap\Omega,
\]
where $M_\varphi^*(f)$ is the bi-parameter non-tangential maximal function of $f$. Finally, from boundedness of $M_\varphi^*$ on $\ltwotwo$ and the smallness of $L^2$-norm of $f$ we get
\[
|\Omega\backslash\Omega'|\le|\{M_\varphi^*(f)(x)>\frac{1}{4}\}|\lesssim \int_{\R^2}|f|^2\lesssim\delta^{\frac{1}{2}}|\Omega|,
\]
and thus if we choose $\delta$ small enough \eqref{2222} holds and this completes the proof.
\end{proof}

\section*{Acknowledgments}
I would like to thank the referee for valuable comments and suggestions on the manuscript.


\end{document}